\documentclass[10pt, conference]{ieeeconf}
\IEEEoverridecommandlockouts
\usepackage[english]{babel}

\usepackage{dsfont} 

\usepackage{url}

\usepackage{amsmath}
\usepackage{amsthm}
\usepackage{amssymb}

\usepackage{marvosym} 

\usepackage{tikz}
\usetikzlibrary{arrows,automata,backgrounds,decorations}

\newcommand{\e}[1]{ {\mathrm{e}}^{ #1 } }

\newcommand{\iterand}[2]{ #1^{[#2]} }
\newcommand{\vect}[1]{ \boldsymbol{#1} }
\newcommand{\vectones}[1]{ \vect{1}_{#1} }
\newcommand{\vectelementary}[2]{ \vect{e}_{#1,#2} }
\newcommand{\vectComponent}[2]{ #1_{#2} }
\newcommand{\vectInLine}[1]{ ( #1 )^{\mathrm{T}} }
\newcommand{\matrixElement}[3]{ {#1}_{#2,#3} }

\newcommand{\matrixRow}[2]{ {\vect{#1}}_{ #2 , \cdot } }
\newcommand{\pnorm}[2]{ \| #1 \|{}_{#2} }

\newcommand{\transpose}[1]{ #1{}^{\mathrm{T}} }
\newcommand{\truncate}[2]{ [ {#1} ]^{#2} }

\newcommand{\gradientOperatorWrt}[1]{ \nabla_{#1} }

\newcommand{\criticalpoint}[1]{  #1^{\textnormal{opt}} }
\newcommand{\elementaryvector}[1]{ \vect{e}_{#1} }

\newcommand{\cardinality}[1]{ | #1 | }

\newcommand{\expectation}[1]{ \mathbb{E} [ #1 ] }

\newcommand{\indicator}[1]{ \mathds{1} [ #1 ] }

\newcommand{\probability}[1]{ \mathbb{P} [ #1 ] }

\newcommand{\process}[2]{ \{ #1 \}_{ #2 } }
\newcommand{\totalVariation}[1]{ || #1 ||_{\mathrm{var}} }

\newcommand{\naturalNumbers}{ \mathbb{N} }
\newcommand{\realNumbers}{ \mathbb{R} }
\newcommand{\positiveRealNumbers}{ [0,\infty) }
\newcommand{\strictlyPositiveRealNumbers}{ (0,\infty) }

\newcommand{\sA}{ {\ensuremath{x}} } 
\newcommand{\sB}{ {\ensuremath{y}} } 
\newcommand{\sC}{ {\ensuremath{z}} } 
\newcommand{\sD}{ {\ensuremath{v}} } 

\newcommand{\svA}{ {\sA} }
\newcommand{\svB}{ {\sB} }
\newcommand{\svC}{ {\sC} }
\newcommand{\svD}{ {\sD} }

\newcommand{\stateSpace}{ \Omega }
\newcommand{\itstep}[1]{\iterand{a}{#1}}
\newcommand{\iterror}[1]{\iterand{e}{#1}}
\newcommand{\itfreq}[1]{\iterand{f}{#1}}
\newcommand{\totaltime}[1]{\iterand{t}{#1}}

\newcommand{\normalizationConstant}{Z}

\newcommand{\LipschitzConstant}{c_{\mathrm{l}}}

\newcommand{\gradientConstant}{c_{\mathrm{g}}}

\newcommand{\QuodEratDemonstrandum}{\hfill \ensuremath{\Box}}

\newcommand{\refEquation}[1]{{\textrm{\eqref{#1}}}}

\newcommand{\refTheorem}[1]{{\textrm{Theorem~\ref{#1}}}}

\newcommand{\refProposition}[1]{{\textrm{Proposition~\ref{#1}}}}
\newcommand{\refLemma}[1]{{\textrm{Lemma~\ref{#1}}}}

\newcommand{\refSection}[1]{{\textrm{\S\ref{#1}}}}
\newcommand{\refAppendixSection}[1]{{\textrm{Appendix \ref{#1}}}}

\newtheorem{theorem}{Theorem}
\newtheorem{proposition}{Proposition}
\newtheorem{lemma}{Lemma}

\begin{document}

\title{\LARGE \bf Achievable Performance in Product-Form Networks}

\author{%
Jaron Sanders \\
\normalsize
Eindhoven University of Technology \\
Eindhoven, The Netherlands \\
jaron.sanders@tue.nl
\and
Sem C. Borst \\
\normalsize
Bell Laboratories, Alcatel-Lucent \\
Murray Hill, United States of America \\
sem@research.bell-labs.com
\and
Johan S.H. van Leeuwaarden \\
\normalsize
Eindhoven University of Technology \\
Eindhoven, The Netherlands \\
j.s.h.v.leeuwaarden@tue.nl
}

\maketitle

\begin{abstract}
We characterize the achievable range of performance measures in product-form networks where one or more system parameters can be freely set by a network operator. Given a product-form network and a set of configurable parameters, we identify which performance measures can be controlled and which target values can be attained. We also discuss an online optimization algorithm, which allows a network operator to set the system parameters so as to achieve target performance metrics.  In some cases, the algorithm can be implemented in a distributed fashion, of which we give several examples. Finally, we give conditions that guarantee convergence of the algorithm, under the assumption that the target performance metrics are within the achievable range.
\end{abstract}

\section{Introduction}

Many stochastic systems are modelled using Markov processes. Models range from communication systems, computer networks and data center applications to content dissemination systems and physical or social interactions processes \cite{Kelly85,Liggett85}. In particular, the framework of reversible Markov processes \cite{Kelly79,Kelly91} allows for an extensive analysis of such systems, often geared towards optimizing performance measures such as sojourn times, queue lengths and holding costs \cite{BKMS87,Bremaud99,WK05}.

Instead of optimizing performance measures, we are interested in identifying which performance measures can be achieved. Any performance measure $\vect{\gamma}$ for which there exist finite parameters $\criticalpoint{\vect{r}}$ such that the performance of the system equals $\vect{\gamma}$, is called an achievable target. The collection of achievable targets is called the achievable region. A parameter is anything that changes a transition rate, such as processing speeds, number of servers and job sizes. 

In this paper, we describe which performance measures can be influenced in product-form networks \cite{BCMP79,JLSSS08,Kelly79,Kelly91}, given a list of configurable parameters. Our work makes explicit that the more configurable parameters one has, the more control one can exert on a system. We also identify the achievable region of these performance measures by assuming that parameters are unbounded, and that in this case the achievable region is a convex hull of a set of vectors. Which vectors there are in this set depends on which parameters there are, as well as which states there are in the state space $\stateSpace$.

Using our analysis of the achievable region, operators can know which performance measures the operator influences when changing parameters. By examining the achievable region of a system, an operator can furthermore know which performance measures are achievable. Supposing that an operator wants its system's performance measures to equal a certain achievable target $\vect{\gamma}$, we then proceed to show how the (distributed) online algorithm in \cite{SandersBorstLeeuwaarden2012} can be used to find $\criticalpoint{\vect{r}}$. Related ideas on using online algorithms to optimize networks can be found in \cite{MarbachTsitsiklis2001,MarbachTsitsiklis2003,Puterman94}.

The online algorithm in \cite{SandersBorstLeeuwaarden2012} is a stochastic gradient algorithm \cite{Borkar08,Cao2007,KY03}, known to converge when operators can set parameters in a compact set and when operators choose appropriate step sizes and observation periods. We note however, that the conditions that guarantee convergence as described in \cite{SandersBorstLeeuwaarden2012} are insufficient when parameters are unbounded. The conditions need to be more stringent in order to avoid extreme parameter growth. We shall capitalize on the proof methodology presented in \cite{SandersBorstLeeuwaarden2012} in order to derive sufficient conditions to guarantee convergence with probability one for the application we have in mind here.

In related work, Jiang and Walrand \cite{JiangWalrand2009,JiangWalrand2010} developed a distributed online algorithm that tells nodes in CSMA/CA networks (transmitter-receiver pairs) how to set their activation rates so that their throughput equals a given target. They also identified the achievable region of the throughput under the assumption that activation rates can be unbounded. This paper generalizes and extends their results to the broad class of product-form networks.

This paper discusses two topics and is organized as follows. The first topic is the achievable region. We describe our model in \refSection{sec:Achievable_region__Model_description} and identify the achievable region in \refSection{sec:Theorem_describing_the_achievable_region} for several examples in \refSection{sec:Achievable_region__Examples}. We also provide an in-depth discussion as to how we identified the achievable region in \refSection{sec:Appendix__Proof_of_achievable_region}. The second topic is finding parameters such that the performance measure of the system attains some target value from within the achievable region. We describe how to use an online algorithm to find these parameters in \refSection{sec:Algorithm_description}, and we provide sufficient conditions to guarantee convergence in \refSection{sec:Algorithm__Conditions_for_convergence}. A proof that these conditions are sufficient is then given in \refSection{sec:Algorithm__Convergence_proof}.
\section{Achievable region}
\label{sec:Achievable_region}

Throughout this paper, we denote by $\vectComponent{b}{i}$ the $i$-th element of vector $\vect{b}$. When taking a scalar function of an $n$-dimensional vector $\vect{b}$, we do this element-wise, i.e.~$\exp{ ( \vect{b} ) } = \vectInLine{ \exp \vectComponent{b}{1}, ..., \exp \vectComponent{b}{n} }$. If we have a $\cardinality{\stateSpace}$-dimensional vector $\vect{b}$ in which each element corresponds to some state $\svA \in \stateSpace$, we write $\vectComponent{b}{\svA}$ for that element of $\vect{b}$ that corresponds to state $\svA$.  Similarly, we denote by $\matrixElement{A}{i}{j}$ the element in row $i$, column $j$ of matrix $A$. If rows and/or columns correspond to states in $\stateSpace$, we write $\matrixElement{A}{\svA}{\svB}$ instead. Finally, we denote by $\vectones{n}$ the $n$-dimensional vector of which all elements equal one and by $\vectelementary{n}{i}$ the $n$-dimensional vector with all of its element equalling zero except for a one in the $i$-th position.

\subsection{Model description}
\label{sec:Achievable_region__Model_description}

Consider an irreducible, reversible Markov process $\process{ X(t) }{ t \geq 0 }$ on a finite state space $\stateSpace$ with generator matrix $Q \in \realNumbers^{ \cardinality{\stateSpace} \times \cardinality{\stateSpace} }$. 
We consider cases where the Markov process models a system in which an operator can change one or more transition rates. We assume that an operator only changes transition rates in such a way that the process remains irreducible and reversible, and to avoid trivialities, we assume that there are $d > 0$ configurable transition rates. We call the logarithm of such a configurable transition rate a parameter $\vectComponent{r}{i}$, where $i = 1, ..., d$, and collect all parameters in the vector $\vect{r} = \vectInLine{ \vectComponent{r}{1}, \vectComponent{r}{2}, ..., \vectComponent{r}{d} }$. In other words for $i = 1, ..., d$, there exist $\svA$, $\svB \in \stateSpace$ such that $\vectComponent{r}{i} = \ln \matrixElement{Q}{\svA}{\svB}$.

Under these assumptions, the process has a steady-state distribution $\vect{\pi}(\vect{r})$ that can be written in the product form
\begin{align}
\vect{\pi}(\vect{r}) 
= \frac{1}{\normalizationConstant(\vect{r})} \exp{ ( A \vect{r} + \vect{b} ) }, \label{eqn:Product_form_distribution__ExpAb_form}
\end{align}
where $A \in \realNumbers^{ \cardinality{\stateSpace} \times d}$ is a matrix, $\vect{r} \in \realNumbers^d$, $\vect{b} \in \realNumbers^{ \cardinality{\stateSpace} }$ are vectors and $\normalizationConstant(\vect{r}) = \transpose{  \vectones{\cardinality{\stateSpace}} } \exp{ ( A \vect{r} + \vect{b} ) }$ is the normalization constant. The matrix $A$ tells us not only which, but also by how much parameters influence steady-state probabilities, while the vector $\vect{b}$ contains all kinds of other constants such as logarithms of rates that are no parameters of the systems. 

When operators change parameters, the steady-state probability distribution changes. In particular, $\transpose{A} \vect{\pi}(\vect{r})$ changes, because the elements
\begin{align}
\vectComponent{ \bigl( \transpose{A} \vect{\pi}(\vect{r}) \bigr) }{i}
= \sum_{\svA \in \stateSpace} \matrixElement{A}{\svA}{i} \vectComponent{\pi}{\svA}(\vect{r}), \textrm{ } i = 1, ..., d,  \label{eqn:Aggregate_steady_state_probability_controlled_by_parameter_i}
\end{align}
are aggregates of steady-state probabilities. These aggregates typically have a physical interpretation. For example, if the service rates $\vectComponent{\mu}{i}$ of queues $i = 1, ..., d$ in a closed Jackson network can be controlled and one defines $\vectComponent{r}{i} = \ln \vectComponent{\mu}{i}$, the right-hand side of \refEquation{eqn:Aggregate_steady_state_probability_controlled_by_parameter_i}  reduces to the (negative) mean number of customers in queue $i$. We come back to this and other examples later, and we then make the analysis explicit.

\subsection{Achievable aggregrate probabilities}
\label{sec:Theorem_describing_the_achievable_region}

Given some vector $\vect{\gamma} \in \realNumbers^d$, we are interested in finding finite parameter values $\criticalpoint{\vect{r}}$ such that $\transpose{A} \vect{\pi}(\criticalpoint{\vect{r}}) = \vect{\gamma}$. We call $\vect{\gamma}$ our \emph{target}. It is not a priori clear whether such values exist, but if they exist and if they are finite, we call the target \emph{achievable}. In this paper, we identify a collection of target vectors that are achievable, which we call the \emph{achievable region} $\mathcal{A}$.

\begin{theorem}
\label{thm:Achievable_region}
Any $\vect{\gamma} \in \mathcal{A} = \bigl\{ \transpose{A} \vect{\alpha} \big| \vect{\alpha} \in (0,1)^{\cardinality{\stateSpace}}, \transpose{ \vect{\alpha} } \vectones{\cardinality{\stateSpace}} = 1 \bigr\}$ is achievable. 

\label{thm:Achievable_region_after_affine_transformation}
Furthermore, if $B \in \mathcal{B} = \realNumbers^{n \times d}$, $n \leq d$, is an affine transformation, there exists finite $\criticalpoint{\vect{r}}$ such that $B \transpose{A} \vect{\pi}(\criticalpoint{\vect{r}}) = \vect{\gamma'}$ for all $\vect{\gamma'} \in \bigl\{ B \transpose{A} \vect{\alpha} \big| \vect{\alpha} \in (0,1)^{\cardinality{\stateSpace}}, \transpose{ \vect{\alpha} } \vectones{\cardinality{\stateSpace}} = 1 \bigr\}$.
\end{theorem}

Note that $\mathcal{A}$ is the interior of the convex hull of all transposed row vectors of $A$. This can be seen by writing
\begin{align}
&\transpose{A} \vect{\alpha}
= \sum_{i=1}^d \vectComponent{ ( \transpose{A} \vect{\alpha} ) }{i} \vectelementary{d}{i}
= \sum_{i=1}^d \sum_{\svA \in \stateSpace} \matrixElement{A}{\svA}{i} \vectComponent{\alpha}{\svA} \vectelementary{d}{i} \nonumber \\
&= \sum_{\svA \in \stateSpace} \vectComponent{\alpha}{\svA} \bigl( \sum_{i=1}^d \matrixElement{A}{\svA}{i} \vectelementary{d}{i} \bigr)
= \sum_{\svA \in \stateSpace} \vectComponent{\alpha}{\svA} \transpose{ \matrixRow{A}{\svA} },
\end{align}
where $\matrixRow{A}{\svA}$ denotes the row vector in matrix $A$ corresponding to state $\svA$. In the special case of modelling a CSMA/CA network, \cite{JiangWalrand2009,JiangWalrand2010} found that the achievable region of the throughput is the interior of the convex hull of all independent sets of an interference graph. We note that if we capture their system in our notation, $A$ would have all independent sets of the interference graph as row vectors.

Before we prove \refTheorem{thm:Achievable_region}, which is the topic of \refSection{sec:Appendix__Proof_of_achievable_region}, we discuss several examples to which \refTheorem{thm:Achievable_region} can be applied.

\subsection{Examples}
\label{sec:Achievable_region__Examples}

\subsubsection{Finite-state birth-and-death process}

Consider a birth-and-death process on $\stateSpace = \{ 0, 1, 2, ..., n \}$. When the system is in state $\sA \in \{ 0, 1, ..., n-1 \}$, it goes to state $\sA + 1$ after an exponentially distributed time with mean $1 / \vectComponent{\lambda}{\sA+1}$. Similarly, when the system is in state $\sA \in \{ 1, ..., n \}$, it goes to state $\sA - 1$ after an exponentially distributed time with mean $1 / \vectComponent{\nu}{\sA}$. 
The steady-state probability of observing the system in state $\sA \in \stateSpace$ is given by $\vectComponent{\pi}{\sA} = \normalizationConstant^{-1} \prod_{i=1}^{\sA} \vectComponent{\lambda}{i} / \vectComponent{\mu}{i}$.

Suppose that we can change $\vectComponent{\lambda}{i}$ for $i=1, ..., n-1$ in this system. With $\vectComponent{r}{i} = \ln \vectComponent{\lambda}{i}$ we know that
\begin{align}
\vectComponent{\pi}{\sA}(\vect{r}) 
= \frac{1}{Z(\vect{r})} \exp{ \Bigl( \sum_{i=1}^{\sA} \vectComponent{r}{i} - \sum_{i=1}^{\sA} \ln \vectComponent{\mu}{i} \Bigr) },
\end{align}
which corresponds to \refEquation{eqn:Product_form_distribution__ExpAb_form} for $\matrixElement{A}{\sA}{i} = \indicator{ \sA \geq i }$ and $\vectComponent{b}{\svA} = - \sum_{i=1}^{\sA} \ln \vectComponent{\mu}{i}$.

For $i=1,...,n$, the right-hand side of \refEquation{eqn:Aggregate_steady_state_probability_controlled_by_parameter_i} expresses the steady-state probability $\probability{ X_t \geq i }$ and using   \refTheorem{thm:Achievable_region}, we can characterize its achievable region. While the achievable region of $\probability{ X_t \geq i }$ and control thereof is interesting, we want to instead determine the achievable region of each steady-state probability $\probability{ X_t = i }$. For this, we define a matrix $B \in \realNumbers^{n \times n}$ element-wise by setting $\matrixElement{B}{r}{c} = \indicator{ r = c } - \indicator{ r = c + 1 }$ for $r,c = 1, ..., n$. We can then use \refTheorem{thm:Achievable_region_after_affine_transformation} to conclude that for any $\vect{\gamma} \in (0,1)^n$ such that $\sum_{i=1}^n \vectComponent{\gamma}{i} = 1$, there exists $\vect{\lambda} \in \strictlyPositiveRealNumbers^n$ so that $\probability{ X_t = i } = \vectComponent{\gamma}{i}$ for $i=1,...,n$. Note that in this example, one has enough (and appropriate) parameters to control the entire steady-state distribution.


\subsubsection{Closed Jackson network}

Consider a closed Jackson network with $d \in \naturalNumbers$ queues and $n \in \naturalNumbers$ permanent customers. A customer that leaves queue $i \in \{ 1, ..., d \}$, joins queue $j \in \{ 1, ..., d \}$ with probability $\matrixElement{P}{i}{j}$. Customers are served at queue $i$ with rate $\vectComponent{\mu}{i}$. The state space is given by $\stateSpace = \{ \svB \in \naturalNumbers^d | \sum_{i=1}^d \vectComponent{\sB}{i} = n \} $, with $\vectComponent{\sA}{i}$ the number of customers in queue $i$. The steady-state probability of observing state $\svA \in \stateSpace$ is $\vectComponent{\pi}{\svA} = \normalizationConstant^{-1} \prod_{i=1}^d ( \vectComponent{\lambda}{i} / \vectComponent{\mu}{i} )^{\vectComponent{\sA}{i}}$. Here, $\vect{\lambda} \in \realNumbers^d$ is a non-zero solution of $\vect{\lambda} = \vect{\lambda} P$. 

Suppose now that we can change $\vectComponent{\mu}{i}$ for $i=1,...,d$. Define $\vectComponent{r}{i} = \ln \vectComponent{\mu}{i}$ for $i=1,...,d$ so that 
\begin{align}
\vectComponent{\pi}{\svA}(\vect{r}) 
= \frac{1}{Z(\vect{r})} \exp{ \Bigl( - \sum_{i=1}^d \vectComponent{\sA}{i} \vectComponent{r}{i} + \sum_{i=1}^d \vectComponent{\sA}{i} \ln \vectComponent{\lambda}{i} \Bigr) }, \label{eqn:Example__Closed_Jackson_network__Product_form_distribution__Exponential_form}
\end{align}
which is equivalent to \refEquation{eqn:Product_form_distribution__ExpAb_form} when $\matrixElement{A}{\svA}{i} = -\vectComponent{\sA}{i}$ and $\vectComponent{b}{\svA} = \sum_{i=1}^d \vectComponent{\sA}{i} \ln \vectComponent{\lambda}{i}$. 

After substituting $\matrixElement{A}{\svA}{i} = -\vectComponent{\sA}{i}$ into \refEquation{eqn:Aggregate_steady_state_probability_controlled_by_parameter_i} we see that the right-hand side of \refEquation{eqn:Aggregate_steady_state_probability_controlled_by_parameter_i} can be interpreted as the negative steady-state expectation of the number of customers in queue $i$. Using \refTheorem{thm:Achievable_region_after_affine_transformation} with $B = -I$, we conclude that for any 
\begin{align}
\vect{\gamma} 
\in &\bigl\{ \sum_{\svA \in \stateSpace} \vectComponent{\alpha}{\svA} \svA \big| \vect{\alpha} \in (0,1)^{\cardinality{\stateSpace}}, \transpose{ \vect{\alpha} } \vectones{\cardinality{\stateSpace}} = 1 \bigr\},
\end{align}
there exists $\vect{\mu} \in \strictlyPositiveRealNumbers^d$ so that the mean stationary queue lengths are given by $\vectComponent{\gamma}{1}, ..., \vectComponent{\gamma}{d}$.

\subsubsection{CSMA network on a partite graph}

Consider the following stylized model for a CSMA network \cite{ZoccaBorstLeeuwaarden2012}. Suppose there are $\vectComponent{n}{k}$ class-$k$ nodes, with $k = 1, ..., K$. If a node is active (transmitting) of say class-$c$, all nodes of classes $k \neq c$ cannot activate (begin transmitting). Class-$c$ nodes, however, can activate. Nodes of class-$k$, $k = 1, ..., K$, try to activate after an exponentially distributed time with mean $1 / \vectComponent{\nu}{k}$. After successfully activating, a node deactivates after an exponentially distributed time with mean $1$.

We will keep track of the number of active nodes and which class they belong to. Specifically, let $(k,l)$ denote the state in which $l$ class-$k$ nodes are active, $l \in \{ 0, 1, ..., \vectComponent{n}{k} \}$ and $k \in \{ 1, ..., K \}$. The equilibrium distribution of this Markov process is given by $\vectComponent{\pi}{(k,l)} = \normalizationConstant^{-1} \tbinom{ \vectComponent{n}{k} }{ l } \vectComponent{\nu}{k}^l$.

Now consider the following two control schemes: (i) the operator can choose any $\nu \in (0,\infty)$ and set $\vectComponent{\nu}{k} = \nu$ for $k = 1, ..., K$, and (ii) the operator can set any $\vect{\nu} \in (0, \infty)^K$. Intuitively, the operator has less control over the network with scheme (i) than with scheme (ii). In practice, however, there could be compelling reasons to use scheme (i) instead of scheme (ii), such as reduced complexity and/or lower operation costs. In such situations, it is worthwhile to examine and compare the achievable regions of all available schemes.

For scheme (i), we identify $r = \ln \nu$ and write
\begin{align}
\vectComponent{\pi}{(k,l)}(r) 
= \frac{1}{\normalizationConstant(r)} \exp{ \bigl( l r + \ln \tbinom{ \vectComponent{n}{k} }{ l } \bigr) }
\end{align}
which is equivalent to \refEquation{eqn:Product_form_distribution__ExpAb_form} when $\vectComponent{A}{(k,l)} = l$ and $\vectComponent{b}{(k,l)} = \ln \binom{ \vectComponent{n}{k} }{ l }$. The right-hand side of \refEquation{eqn:Aggregate_steady_state_probability_controlled_by_parameter_i} can then be interpreted as the steady-state average number of active nodes. Using \refTheorem{thm:Achievable_region}, we conclude that for every $\gamma \in (0, \max_{k = 1, ..., K} \vectComponent{n}{k} )$, there exists a $\nu \in (0,\infty)$ such that $\sum_{k=1}^K \sum_{l=1}^{\vectComponent{n}{k}} l \vectComponent{\pi}{(k,l)} = \gamma$. In other words, by using control scheme (i) and setting $\nu$ appropriately, the average number of active nodes can be made to equal anything between $0$ and $\max_{k = 1, ..., K} \vectComponent{n}{k}$.

In the case of scheme (ii), identify $\vectComponent{r}{i} = \ln \vectComponent{\nu}{i}$ for $i = 1, ..., K$, so that
\begin{align}
\vectComponent{\pi}{(k,l)}(\vect{r}) 
= \frac{1}{\normalizationConstant(\vect{r})} \exp{ \bigl( l \vectComponent{r}{k} + \ln \tbinom{ \vectComponent{n}{k} }{ l } \bigr) }.
\end{align}
Again if we compare to \refEquation{eqn:Product_form_distribution__ExpAb_form}, we see that $\matrixElement{A}{(k,l)}{i} = l \indicator{ k = i }$ and $\vectComponent{b}{(k,l)} = \ln \tbinom{ \vectComponent{n}{k} }{ l }$. Furthermore, \refEquation{eqn:Aggregate_steady_state_probability_controlled_by_parameter_i} is to be interpreted as the steady-state average number of active class-$i$ nodes, where $i=1,...,k$. Using \refTheorem{thm:Achievable_region}, we conclude that for every 
\begin{align}
\vect{\gamma} 
&\in \bigl\{ \sum_{i=1}^K \sum_{l=1}^{\vectComponent{n}{i}} l \vectComponent{\alpha}{(i,l)} \elementaryvector{i} | \vect{\alpha} \in (0,1)^{\cardinality{\stateSpace}}, \transpose{ \vect{\alpha} } \vectones{\cardinality{\stateSpace}} = 1 \bigr\} \nonumber \\
&= \bigl\{ \sum_{k=1}^K \vectComponent{\beta}{k} \vectComponent{n}{k} \elementaryvector{k} | \vect{\beta} \in (0,1)^{K}, \transpose{ \vect{\beta} } \vectones{K} \leq 1 \bigr\}, \label{eqn:CSMA_network_partite_graph__Achievable_region_Scheme_ii}
\end{align}
there exists a $\vect{\nu} \in (0,\infty)^K$ such that $\sum_{l=1}^{\vectComponent{n}{i}} l \vectComponent{\pi}{(i,l)} = \vectComponent{\gamma}{i}$ for $i = 1, ..., K$. We see that by choosing $\vect{\nu}$ appropriately, the average number of active nodes \emph{of every class} can be controlled. As expected, more control can be exerted on the network with scheme (ii) than with scheme (i).

The achievable region given in \refEquation{eqn:CSMA_network_partite_graph__Achievable_region_Scheme_ii} can intuitively be understood as follows. When precisely one element of $\vect{\nu}$ becomes extremely large, say $\vectComponent{\nu}{k} \rightarrow \infty$, all class-$k$ nodes are active almost always and the average number of active nodes would be $\vectComponent{n}{k} \elementaryvector{k}$. If two or more elements become extremely large, say $\vectComponent{\nu}{k_1}, \vectComponent{\nu}{k_2} \rightarrow \infty$, a large number of class-$k_1$ and class-$k_2$ nodes would be active for a large fraction of time (but never simultaneously). The average number of active nodes is then a weighted average of $\vectComponent{n}{k_1} \elementaryvector{k_1}$ and $\vectComponent{n}{k_2} \elementaryvector{k_2}$.

\begin{figure}[!hbtp]
\begin{center}
\begin{tikzpicture}[ ->, >=stealth', shorten >=1pt, auto, node distance=2.25cm, semithick, sloped ]
\tikzstyle{every state} = [ fill=white, draw=black, thick, text=black, scale=0.625]
\footnotesize
\node[state] (s20) {$~0~$};
\node[state] (s10) [above of=s20, draw=none] {};
\node[state] (s30) [below of=s20, draw=none] {};
\node[state] (s11) [right of=s10] {$1,1$};
\node[state] (s21) [right of=s20] {$2,1$};
\node[state] (s31) [right of=s30] {$3,1$};
\node[state] (s12) [right of=s11] {$1,2$};
\node[state] (s22) [right of=s21] {$2,2$};
\node[state] (s32) [right of=s31] {$3,2$};
\node[state] (s23) [right of=s22] {$2,3$};
\node[state] (s33) [right of=s32] {$3,3$};
\node[state] (s24) [right of=s23] {$2,4$};
\node[state] (s25) [right of=s24] {$2,5$};
\path (s20) edge  [bend right=-25] node[below] {$2 \vectComponent{\nu}{1}$} (s11);
\path (s11) edge  [bend right=45] node[above] {$1$} (s20);
\path (s11) edge  [bend right=10] node[below] {$\vectComponent{\nu}{1}$} (s12);
\path (s12) edge  [bend right=10] node[above] {$2$} (s11);
\path (s20) edge  [bend right=10] node[below] {$5 \vectComponent{\nu}{2}$} (s21);
\path (s21) edge  [bend right=10] node[above] {$1$} (s20);
\path (s21) edge  [bend right=10] node[below] {$4 \vectComponent{\nu}{2}$} (s22);
\path (s22) edge  [bend right=10] node[above] {$2$} (s21);
\path (s22) edge  [bend right=10] node[below] {$3 \vectComponent{\nu}{2}$} (s23);
\path (s23) edge  [bend right=10] node[above] {$3$} (s22);
\path (s23) edge  [bend right=10] node[below] {$2 \vectComponent{\nu}{2}$} (s24);
\path (s24) edge  [bend right=10] node[above] {$4$} (s23);
\path (s24) edge  [bend right=10] node[below] {$\vectComponent{\nu}{2}$} (s25);
\path (s25) edge  [bend right=10] node[above] {$5$} (s24);
\path (s20) edge  [bend right=45] node[below] {$3 \vectComponent{\nu}{3}$} (s31);
\path (s31) edge  [bend right=-25] node[above] {$1$} (s20);
\path (s31) edge  [bend right=10] node[below] {$2 \vectComponent{\nu}{3}$} (s32);
\path (s32) edge  [bend right=10] node[above] {$2$} (s31);
\path (s32) edge  [bend right=10] node[below] {$\vectComponent{\nu}{3}$} (s33);
\path (s33) edge  [bend right=10] node[above] {$3$} (s32);
\end{tikzpicture}
\end{center}
\caption{\textrm{Process describing a CSMA network on a partite graph with $K=3$ classes and $\vect{n} = \vectInLine{2,5,3}$.}}
\end{figure}

\subsection{Proof of \refTheorem{thm:Achievable_region}}
\label{sec:Appendix__Proof_of_achievable_region}

In this section, we prove \refTheorem{thm:Achievable_region}. Our method is based on \cite{JiangWalrand2009,JiangWalrand2010}, where the achievable region of the throughput of nodes (transmitter-receiver pairs) in a CSMA/CA network has been determined. We apply the approach to the broader class of product-form networks and provide all necessary adaptations, which leads to \refTheorem{thm:Achievable_region}. A proof sketch is as follows. First, we construct a convex minimization problem in $\vect{r}$ for every vector $\vect{\alpha} \in (0,1)^{\cardinality{\stateSpace}}$ such that $\transpose{\vect{1}} \vect{\alpha} = 1$. This minimization problem is constructed such that in the minimum located at $\criticalpoint{\vect{r}}$, $\transpose{A} \vect{\pi}(\criticalpoint{\vect{r}}) = \transpose{A} \vect{\alpha}$. Next, we show that strong duality holds and that the dual problem of our minimization problem attains its optimal value. This then implies that the target $\vect{\gamma} = \transpose{A} \vect{\alpha}$ is achievable.

We now proceed with the proof. Let $\ln \vect{x} = \vectInLine{ \ln \vectComponent{x}{1}, ..., \ln \vectComponent{x}{n} }$ for $\vect{x} \in \realNumbers^n$ and define the log-likelihood function $u(\vect{r}) = - \transpose{\vect{\alpha}} \ln \vect{\pi}(\vect{r})$, with $\vect{\alpha} \in (0,1)^{\cardinality{\stateSpace}}$ and $\transpose{\vect{1}} \vect{\alpha} = 1$. After substituting \refEquation{eqn:Product_form_distribution__ExpAb_form}, we find that
\begin{align}
u(\vect{r}) 
= \ln Z(\vect{r}) - \transpose{\vect{\alpha}} ( A \vect{r} + \vect{b} ). \label{eqn:Log_likelihood_function}
\end{align}
Let $\vect{g}(\vect{r}) = \nabla_{\vect{r}} u(\vect{r})$ with $\nabla_{\vect{r}} = \vectInLine{ \partial / \partial \vectComponent{r}{1}, ..., \partial / \partial \vectComponent{r}{d} }$, so that
\begin{align}
\vect{g}(\vect{r}) = \transpose{A} \vect{\pi}(\vect{r}) - \transpose{A} \vect{\alpha}. \label{eqn:Derivative_of_g_as_vector_difference}
\end{align}
The log-likelihood function in \refEquation{eqn:Log_likelihood_function} has the aforementioned properties that we are interested in.

\begin{proposition}
\label{prop:Properties_of_the_log_likelihood_function}
The function $u(\vect{r}) = - \transpose{\vect{\alpha}} \ln \vect{\pi}(\vect{r})$ has the properties that \textnormal{(i)} $\inf_{\vect{r} \in \realNumbers^d} u(\vect{r}) \geq 0$, \textnormal{(ii)} $u(\vect{r})$ is continuous in $\vect{r}$, \textnormal{(iii)} $u(\vect{r})$ is convex in $\vect{r}$ and \textnormal{(iv)} at the critical point $\criticalpoint{\vect{r}}$ for which $g(\criticalpoint{\vect{r}}) = \vect{0}$, $\transpose{A} \vect{\pi}(\criticalpoint{\vect{r}}) = \transpose{A} \vect{\alpha}$.
\end{proposition}

\begin{proof}
\textnormal{(i)} By irreducibility, we have that $\vectComponent{\pi}{\svA} > 0$ for all $\svA \in \stateSpace$. Because $\vectComponent{\alpha}{\svA} > 0$ for all $\svA \in \stateSpace$, $\inf_{\vect{r} \in \realNumbers^d} u(\vect{r}) = \inf_{\vect{r} \in \realNumbers^d} - \transpose{\vect{\alpha}} \ln \vect{\pi}(\vect{r}) \geq 0$.
\textnormal{(ii)} Being a composition of continuous functions in $\vect{r}$, $u(\vect{r})$ is continuous in $\vect{r}$.
\textnormal{(iii)} Being a composition of a convex log-sum-exp function $v(\vect{s}) = \ln \bigl( \sum_{\svA \in \stateSpace} \exp{\vectComponent{s}{\svA}} \bigr) - \transpose{\vect{\alpha}} \vect{s}$ and an affine transformation $A\vect{r} + \vect{b}$, $u(\vect{r}) = v( A\vect{r} + \vect{b} )$ is convex in $\vect{r}$ \cite{BoydVandenberghe2004}. 
\textnormal{(iv)} This follows after equating \refEquation{eqn:Derivative_of_g_as_vector_difference} to $\vect{0}$.
\end{proof}

We next prove that there exists a \emph{finite} $\criticalpoint{\vect{r}}$ for which $g(\criticalpoint{\vect{r}}) = \vect{0}$. Using \refProposition{prop:Properties_of_the_log_likelihood_function}, we conclude that if such a vector exists, then $\transpose{A} \vect{\pi}(\criticalpoint{\vect{r}}) = \transpose{A} \vect{\alpha}$. In other words, the target vector $\transpose{A} \vect{\alpha}$ is attained by setting $\vect{r} = \criticalpoint{\vect{r}}$.

Consider the following optimization problem, which we call the primal problem. 
\begin{align}
\begin{array}{cc}
\textrm{maximize} & -\transpose{ \vect{\beta} } \ln \vect{\beta} + \transpose{ ( \vect{\beta} - \vect{\alpha} )} \vect{b}, \\
\textnormal{over} & \vect{\beta} \in (0,1)^{\cardinality{\stateSpace}}, \\
\textnormal{subject to} & \transpose{A} \vect{\beta} = \transpose{A} \vect{\alpha}, \transpose{\vect{1}} \vect{\beta} = 1. \\
\end{array}
\label{eqn:Primal_problem}
\end{align}
It has been constructed in such a way that its dual is the minimization of $u(\vect{r})$ over $\vect{r}$. This is by design, and we will use and prove this throughout the remainder of this section. We note first that \refEquation{eqn:Primal_problem} differs from the minimization problem considered in \cite{JiangWalrand2009,JiangWalrand2010}, in that there is a second term $\transpose{ ( \vect{\beta} - \vect{\alpha} )} \vect{b}$ necessary to capture the broader class of product-form networks and to allow for the selection of parameters.

Before we can prove that the minimization of $u(\vect{r})$ over $\vect{r}$ is indeed the dual to \refEquation{eqn:Primal_problem}, we need to verify that strong duality holds.

\begin{lemma}
Strong duality holds.
\end{lemma}

\begin{proof}
Slater's condition \cite{BoydVandenberghe2004} tells us that strong duality holds if there exists a $\vect{\beta}$ such that all constraints hold. Considering $\vect{\beta} = \vect{\alpha}$ completes the proof.
\end{proof}
Strong duality implies that the optimality gap is zero, specifically implying that if \refEquation{eqn:Primal_problem} has a finite optimum, its dual also attains a finite optimum. 

\begin{proposition}
\label{prop:Optimal_value_of_the_dual_problem_is_attained}
The optimal value of the dual problem is attained.
\end{proposition}

\begin{proof}
Assume that the optimal solution $\vect{\chi}$ of \refEquation{eqn:Primal_problem} is such that $\vectComponent{\chi}{\svA} = 0$ for all $\svA$ in $\mathcal{I}$. Being the optimal solution, $\vect{\chi}$ must be feasible. Recall that $\vect{\alpha}$ is feasible by assumption. Any distribution on the line that connects $\vect{\alpha}$ and $\vect{\chi}$ is therefore also feasible. When moving from $\vect{\chi}$ towards $\vect{\alpha}$, thus along the direction $\vect{\alpha} - \vect{\chi}$, the change of the objective function in \refEquation{eqn:Primal_problem} is proportional to
\begin{align}
&\transpose{ ( \vect{\alpha} - \vect{\chi} ) } \nabla_{\vect{\beta}} \bigl( -\transpose{ \vect{\beta} } \ln \vect{\beta} + \transpose{ ( \vect{\beta} - \vect{\alpha} )} \vect{b} \bigr) \big|_{\vect{\beta} = \vect{\chi}} \nonumber \\
= &\transpose{ ( \vect{\alpha} - \vect{\chi} ) } \bigl( - \ln \vect{\chi} - \vect{1} + \vect{b} \bigr). \label{eqn:Change_of_objective_function_when_moving_towards_the_optimal_solution}
\end{align}
For $\svA \not\in \mathcal{I}$, $\vectComponent{\chi}{\svA} > 0$ so that $- \ln \vectComponent{\chi}{\svA} - 1 + \vectComponent{b}{\svA}$ is finite. For $\svA \in \mathcal{I}$, $\vectComponent{\chi}{\svA} > 0$ so that \refEquation{eqn:Change_of_objective_function_when_moving_towards_the_optimal_solution}  equals $\infty$ (recall that $\vectComponent{\alpha}{\svA} > 0$ for all $\svA \in \stateSpace$). This means that the objective function increases when moving $\vect{\beta}$ away from $\vect{\chi}$ towards $\vect{\alpha}$. It is therefore not possible that $\vect{\chi}$ is the optimal solution, contradicting our assumption. 

The optimal solution must be such that $\vectComponent{\beta}{\svA} > 0$ for all $\svA \in \stateSpace$. This implies that $-\transpose{ \vect{\beta} } \ln \vect{\beta} + \transpose{ ( \vect{\beta} - \vect{\alpha} )} \vect{b} < \infty$ and using Slater's condition \cite{BoydVandenberghe2004}, we find that the optimal value of the dual problem is attained.
\end{proof}

Strong duality also implies that there exist finite dual variables so that the Lagrangian is maximized. These dual variables turn out to be precisely $\criticalpoint{\vect{r}}$.
So what remains is to show that the dual problem to \refEquation{eqn:Primal_problem} is indeed the minimization of $u(\vect{r})$ over $\vect{r}$ and that the finite dual variables for which the Lagrangian is maximized are indeed $\criticalpoint{\vect{r}}$.

\begin{proposition}
\label{prop:Dual_problem}
The dual problem to \refEquation{eqn:Primal_problem} is 
\begin{align}
\begin{array}{cc}
\textnormal{minimize} & u(\vect{r}), \\
\textnormal{over} & \vect{r} \in \realNumbers^d. \\
\end{array}
\label{eqn:Dual_problem}
\end{align}
\end{proposition}

\begin{proof}
By strong duality, we know that there exist finite dual variables $\criticalpoint{\vect{r}} \in \realNumbers^d$, $\criticalpoint{\vect{w}} \in \positiveRealNumbers^{\cardinality{\stateSpace}}$ and $\criticalpoint{z} \in \realNumbers$ such that the Lagrangian
\begin{align}
&L(\vect{\beta}; \criticalpoint{\vect{r}}, \criticalpoint{\vect{w}}, \criticalpoint{z}) 
= -\transpose{ \vect{\beta} } \ln \vect{\beta} + \transpose{ ( \vect{\beta} - \vect{\alpha} ) } \vect{b} \nonumber \\
&+ \transpose{ ( \transpose{A} \vect{\beta} - \transpose{A} \vect{\alpha} ) } \criticalpoint{\vect{r}} + \transpose{\vect{\beta}} \criticalpoint{\vect{w}} + ( \transpose{\vect{1}} \vect{\beta} - 1 ) \criticalpoint{z}
\end{align}
is maximized by the optimal solution $\criticalpoint{\vect{\beta}}$. By complementary slackness, $\criticalpoint{\vect{w}} = \vect{0}$. Because $\criticalpoint{\vect{\beta}}$ maximizes the Lagrangian,
\begin{align}
\nabla_{\vect{\beta}} L(\criticalpoint{\vect{\beta}}; \criticalpoint{\vect{r}}, \criticalpoint{\vect{w}}, \criticalpoint{z})
= &- \ln \criticalpoint{\vect{\beta}} - \vect{1} \\
&+ \vect{b} + A \criticalpoint{\vect{r}} + \criticalpoint{z} \vect{1} = \vect{0}. \nonumber
\end{align}
This equation can be solved for $\criticalpoint{\vect{\beta}}$, resulting in
$\criticalpoint{\vect{\beta}} =  \exp{ \bigl( (\criticalpoint{z}-1) \vect{1} + A \criticalpoint{\vect{r}} + \vect{b} \bigr) }$. The constant $\criticalpoint{z}$ follows from the normalizing condition $\transpose{\vect{1}} \vect{\beta} = 1$, implying that
\begin{align}
\criticalpoint{\vect{\beta}} 
= \frac{1}{\normalizationConstant(\criticalpoint{\vect{r}})} \exp{ \bigl( A \criticalpoint{\vect{r}} + \vect{b} \bigr) },
\end{align}
where $\normalizationConstant(\criticalpoint{\vect{r}}) = \transpose{\vect{1}} \exp{ \bigl( A \criticalpoint{\vect{r}} + \vect{b} \bigr) }$.

Because $\criticalpoint{\vect{\beta}}$ is the optimal solution, we have that
\begin{align}
&\max_{ \vect{\beta} \in (0,1)^{\cardinality{\stateSpace}} } L(\vect{\beta}; \criticalpoint{\vect{r}}, \criticalpoint{\vect{w}}, \criticalpoint{z}) 
= L(\criticalpoint{\vect{\beta}}; \criticalpoint{\vect{r}}, \criticalpoint{\vect{w}}, \criticalpoint{z}) \nonumber \\
&= - \transpose{ {\criticalpoint{\vect{\beta}}} } \ln \criticalpoint{\vect{\beta}} + \transpose{ ( \criticalpoint{\vect{\beta}} - \vect{\alpha} ) } \vect{b} + \transpose{ ( \transpose{A} \criticalpoint{\vect{\beta}} - \transpose{A} \vect{\alpha} ) } \criticalpoint{\vect{r}} \nonumber \\
&= - \transpose{ {\criticalpoint{\vect{\beta}}} } ( \ln \criticalpoint{\vect{\beta}} - A \criticalpoint{\vect{r}} - \vect{b} ) - \transpose{\vect{\alpha}} ( A \criticalpoint{\vect{r}} + \vect{b} ) \nonumber \\
&= \ln \normalizationConstant(\criticalpoint{\vect{r}}) - \transpose{\vect{\alpha}} ( A \criticalpoint{\vect{r}} + \vect{b} )
= u(\criticalpoint{\vect{r}}).
\end{align}
Because $\criticalpoint{\vect{\beta}}$ and $\criticalpoint{\vect{r}}$ solve the primal problem \refEquation{eqn:Primal_problem}, $\criticalpoint{\vect{r}}$ is the solution of $\min_{\vect{r} \in \realNumbers^d} u(\vect{r})$.
\end{proof}

This concludes the proof of \refTheorem{thm:Achievable_region}.
\section{Algorithm}
\label{sec:Algorithm}

Having identified the achievable region in \refSection{sec:Achievable_region}, we now turn to developing an algorithm that finds $\criticalpoint{\vect{r}}$. For this, we modify a (distributed) online algorithm developed and discussed in \cite{SandersBorstLeeuwaarden2012}. 

\subsection{Description}
\label{sec:Algorithm_description}

We apply the algorithm to the objective function $u(\vect{r}) = - \transpose{\vect{\alpha}} \ln \vect{\pi}(\vect{r})$, where $\vect{\alpha} \in (0,1)^{\cardinality{\stateSpace}}$ is such that $\transpose{\vect{1}} \vect{\alpha} = 1$ and $\transpose{A} \vect{\alpha} = \vect{\gamma}$. When $\vect{\gamma}$ is achievable, we know that such a distribution exists by \refTheorem{thm:Achievable_region}. From \refProposition{prop:Properties_of_the_log_likelihood_function}, we see that $u(\vect{r})$ is convex in $\vect{r}$ and that the gradient $\nabla_{\vect{r}} u(\vect{r}) = \transpose{A} \vect{\pi}(\vect{r}) - \vect{\gamma}$. In its unique critical point $\criticalpoint{\vect{r}}$ where $\nabla_{\vect{r}} u(\criticalpoint{\vect{r}}) = \vect{0}$, we indeed have that $\transpose{A} \vect{\pi}(\criticalpoint{\vect{r}}) = \vect{\gamma}$. The algorithm will find this critical point. 

The algorithm is as follows. Let $0 = \totaltime{0} < \totaltime{1} < ...$ and take initial parameters $\vect{r} = \iterand{\vect{R}}{0}$. At time $\totaltime{n+1}$, marking the end of observation period $n+1$, calculate 
\begin{align}
\iterand{ \vectComponent{\hat{\Pi}}{\svA} }{n+1}
= \frac{1}{ \totaltime{n+1} - \totaltime{n} } \int_{ \totaltime{n} }^{ \totaltime{n+1} } \indicator{ \iterand{Z}{n}(t) = \svA } dt
\end{align}
for every state $\svA \in \stateSpace$. Here, $\process{ \iterand{Z}{n}(t) }{ \totaltime{n} \leq t \leq \totaltime{n+1} }$ is a time-homogeneous Markov process, which starts in $\iterand{Z}{n-1}( \totaltime{n} )$ and evolves according to the generator of $\process{X(t)}{t \geq 0}$ that corresponds to parameters $\iterand{\vect{R}}{n}$. Then update all parameters by setting 
\begin{align}
\iterand{\vect{R}}{n+1} 
= \truncate{ \iterand{\vect{R}}{n} - \itstep{n+1} ( \transpose{A} \iterand{ \hat{\vect{\Pi}}}{n+1} - \vect{\gamma} ) }{ \mathcal{R} }. \label{eqn:Online_algorithm_for_loglikelihood_for_general_R}
\end{align}
Here, $\itstep{n+1}$ denotes the step size and the truncation $\truncate{ \vect{r} }{ \mathcal{R} }$ is component-wise defined as 
\begin{align}
\vectComponent{ \truncate{ \vect{r} }{ \mathcal{R} } }{i} 
= \max \{ \vectComponent{\mathcal{R}}{i}^{\min}, \min \{ \vectComponent{\mathcal{R}}{i}^{\max}, \vectComponent{r}{i} \} \}
\end{align}
for $\mathcal{R} = [ \vectComponent{ \mathcal{R}^{\min} }{1}, \vectComponent{ \mathcal{R}^{\max} }{1} ] \times ... \times [ \vectComponent{ \mathcal{R}^{\min} }{d}, \vectComponent{ \mathcal{R}^{\max} }{d} ]$, which denotes the set of available parameter values.

Conditions on $\itstep{n}$ and $\itfreq{n} = 1 / ( \totaltime{n} - \totaltime{n-1} )$ that ensure convergence of \refEquation{eqn:Online_algorithm_for_loglikelihood_for_general_R} can be found in \cite{SandersBorstLeeuwaarden2012}, but only when $0 < \vectComponent{ \mathcal{R}^{\max} }{d} - \vectComponent{ \mathcal{R}^{min} }{d} < \infty$ for $i = 1, ..., d$. Determining the achievable region for this set of available parameters requires solving a potentially NP-hard inversion problem \cite{LMK94}. Instead, we would like to use \refTheorem{thm:Achievable_region}, but \refTheorem{thm:Achievable_region} only holds under the assumption that $[ \vectComponent{ \mathcal{R}^{\min} }{i}, \vectComponent{ \mathcal{R}^{\max} }{i} ] = \realNumbers$ for $i = 1, ..., d$. 

We therefore need to establish new conditions on $\itstep{n}$ and $\itfreq{n}$ that ensure convergence for the case that $\mathcal{R} = \realNumbers^d$. Such conditions can be found by modifying the proof in \cite[\S{IV}]{SandersBorstLeeuwaarden2012}, which is the topic of the remainder of this section.

\subsection{Conditions for convergence}
\label{sec:Algorithm__Conditions_for_convergence}

When $\mathcal{R} = \realNumbers^d$, we write
\begin{align}
\iterand{\vect{R}}{n+1} 
= \iterand{\vect{R}}{n} - \itstep{n+1} ( \transpose{A} \iterand{ \hat{\vect{\Pi}}}{n+1} - \vect{\gamma} ). \label{eqn:Online_algorithm_for_loglikelihood_for_Rd}
\end{align}
For the algorithm in \refEquation{eqn:Online_algorithm_for_loglikelihood_for_Rd}, we will prove following result.

\begin{theorem}
\label{thm:Convergence}
The sequence $\iterand{\vect{R}}{n}$ generated by the online algorithm \refEquation{eqn:Online_algorithm_for_loglikelihood_for_Rd} converges to the optimal solution $\criticalpoint{\vect{r}}$ with probability one, if $\itstep{n}$, $\iterror{n}$ and $\itfreq{n}$ are such that
\begin{itemize}
\item[\textnormal{(i)}] $\sum_{n=1}^{\infty} \itstep{n} = \infty$, $\sum_{n=1}^{\infty} (\itstep{n})^2 < \infty$,
\item[\textnormal{(ii)}] $\sum_{n=1}^{\infty} \itstep{n} \bigl( \sum_{m=1}^{n-1} \itstep{m} \bigr) \bigl( \iterror{n} + \exp ( c_{2} \sum_{m=1}^n \itstep{m} - c_3 \frac{ (\iterror{n})^2 }{ \itfreq{n} } \exp{ ( - c_4 \sum_{m=1}^n \itstep{m} ) } ) \bigr) < \infty$ for all $c_2$, $c_3$, $c_4 \in \strictlyPositiveRealNumbers$.
\end{itemize}
\end{theorem}

In \refProposition{prop:Sequences_that_meet_the_conditions}, we give two sets of sequences $\itstep{n}$, $\iterror{n}$ and $\itfreq{n}$, labelled \textnormal{(a)} and \textnormal{(b)}, such that conditions \textnormal{(i)} and \textnormal{(ii)} in \refTheorem{thm:Convergence} hold. The proofs that these sequences indeed satisfy the conditions are deferred to \refAppendixSection{sec:Proof_of__Sequences_that_meet_the_conditions}. The choices made for the step sizes and observation periods in \textnormal{(a)} are based on similar ideas as in \cite{JiangWalrand2009,JiangWalrand2010}. Note however that conditions \textnormal{(i)} and \textnormal{(ii)} differ from the conditions in \cite{JiangWalrand2009,JiangWalrand2010}, and the verification that \textnormal{(i)} and \textnormal{(ii)} hold is therefore different.
\begin{proposition}
\label{prop:Sequences_that_meet_the_conditions}
Conditions \textnormal{(i)} and \textnormal{(ii)} in \refTheorem{thm:Convergence} hold for
\begin{itemize}
\item[\textnormal{(a)}] $\itstep{n} = ( n \ln(n+1) )^{-1}$, $\iterror{n} = n^{-\alpha/2}$ and $\itfreq{n} = n^{-\delta} $ where $\alpha > 0$ and $\delta > 1 + \alpha$, and
\item[\textnormal{(b)}] $\itstep{n} = n^{-1}$, $\iterror{n} = ( n^{\alpha/2} \sum_{m=1}^{n-1} m^{-1} )^{-1}$ and $\itfreq{n} = ( \ln n + 1 )^{-2} n^{- \delta }$ where $\alpha > 0$ and $\delta \geq 1 + \alpha + c_4 = 1 + \alpha + \gradientConstant ( 1 + 2 \max_{\svB \in \stateSpace} \pnorm{ \matrixRow{A}{\svB} }{1} )$.
\end{itemize}
\end{proposition}

\subsection{Convergence proof}
\label{sec:Algorithm__Convergence_proof}

We start by defining the error bias $\iterand{\vect{B}}{n} = \expectation{ \iterand{ \vect{\hat{G}} }{n} | \iterand{\mathcal{F}}{n-1} } - \iterand{ \vect{G} }{n}$ and zero-mean noise $\iterand{\vect{E}}{n} = \iterand{ \vect{\hat{G}} }{n} - \expectation{ \iterand{ \vect{\hat{G}} }{n} | \iterand{\mathcal{F}}{n-1} }$, respectively. Here, $\iterand{\mathcal{F}}{n-1}$ denotes the $\sigma$-field generated by the random vectors $\iterand{\vect{Z}}{0}, \iterand{\vect{Z}}{1}, ..., \iterand{\vect{Z}}{n-1}$, where $\iterand{\vect{Z}}{0} = \vectInLine{ \iterand{\vect{R}}{0}, X(0) }$ and $\iterand{\vect{Z}}{n} = \vectInLine{ \iterand{\vect{\hat{G}}}{n}, \iterand{\vect{R}}{n}, X({\totaltime{n}}) }$ for $n \geq 1$. We will proceed to use \refLemma{lem:Conditions_for_convergence}, proven in \cite[\S{4.1}]{SandersBorstLeeuwaarden2012}

\begin{lemma}
\label{lem:Conditions_for_convergence}
The sequence $\iterand{\vect{R}}{n}$ generated by either online algorithm \refEquation{eqn:Online_algorithm_for_loglikelihood_for_general_R} or \refEquation{eqn:Online_algorithm_for_loglikelihood_for_Rd} converges to the optimal solution $\criticalpoint{\vect{r}}$ with probability one, if the summation $\sum_{n=1}^{\infty} \itstep{n} \bigl| \expectation{ \transpose{ \iterand{\vect{B}}{n} } ( \iterand{\vect{R}}{n-1} - \criticalpoint{\vect{r}} ) | \iterand{\mathcal{F}}{n-1} } \bigr|$ is finite with probability one and if also, for any $\varepsilon > 0$, there exists $m_0 \in \naturalNumbers$ so that for any $n \geq m \geq m_0$,
\begin{itemize}
\item[\textnormal{(i)}] $\sum_{j=m}^{n} \itstep{j} \transpose{ \iterand{\vect{B}}{j} } ( \criticalpoint{\vect{r}} - \iterand{\vect{R}}{j-1} ) \leq \varepsilon$ and
\item[\textnormal{(ii)}] $\sum_{j=m}^{n} \itstep{j} \transpose{ \iterand{\vect{E}}{j} } ( \criticalpoint{\vect{r}} - \iterand{\vect{R}}{j-1} ) \leq \varepsilon$
\end{itemize}
with probability one.
\end{lemma}

Proving \refTheorem{thm:Convergence} thus boils down to verifying the conditions in \refLemma{lem:Conditions_for_convergence} for \refEquation{eqn:Online_algorithm_for_loglikelihood_for_Rd}. This is however more complicated than for \refEquation{eqn:Online_algorithm_for_loglikelihood_for_general_R}, which has been done in \cite[\S{IV-B}]{SandersBorstLeeuwaarden2012}. This is because it is possible for a component of $\iterand{\vect{R}}{n+1}$ to run off to $\pm \infty$. Fortunately, because $\pnorm{ \gradientOperatorWrt{\vect{r}} u(\vect{r}) }{2} \leq \gradientConstant = \cardinality{\stateSpace} d \max_{\svA,i} | \matrixElement{A}{\svA}{i} |$ \cite[\S{III-B}]{SandersBorstLeeuwaarden2012}, the rate at which the algorithm can be bounded.
\begin{lemma}
\label{lem:Upper_bound_for_components_of_R_n}
For $n \in \naturalNumbers$ and $i = 1, ..., d$,
$| \vectComponent{ \iterand{R}{n} }{i} | \leq | \vectComponent{ \iterand{R}{0} }{i} | + \gradientConstant \sum_{m=1}^{n} \itstep{m}$.
\end{lemma}

\begin{proof}
Fix $n \in \naturalNumbers$ and use the triangle inequality repeatedly to obtain
\begin{align}
&| \vectComponent{ \iterand{R}{n} }{i} |
= | \vectComponent{ \iterand{R}{n-1} }{i} - \itstep{n} \vectComponent{ \iterand{\hat{G}}{n} }{i} |
\leq | \vectComponent{ \iterand{R}{n-1} }{i} | + \itstep{n} | \vectComponent{ \iterand{\hat{G}}{n} }{i} | \nonumber \\
&\leq | \vectComponent{ \iterand{R}{0} }{i} | + \sum_{m=1}^{n} \itstep{m} | \vectComponent{ \iterand{\hat{G}}{m} }{i} |
\leq | \vectComponent{ \iterand{R}{0} }{i} | + \gradientConstant \sum_{m=1}^{n} \itstep{m}.
\end{align}
This completes the proof.
\end{proof}

We now turn to verifying \refLemma{lem:Conditions_for_convergence}. To do so, we consider the error bias and zero-mean noise separately, similar to \cite[\S{IV-B}]{SandersBorstLeeuwaarden2012}.

\subsubsection{Error bias} 

Using \refLemma{lem:Upper_bound_for_components_of_R_n}, we find that
\begin{align}
&\sum_{n=1}^{\infty} \itstep{n} \bigl| \expectation{ \transpose{ \iterand{\vect{B}}{n} } ( \iterand{\vect{R}}{n-1} - \criticalpoint{\vect{r}} ) | \iterand{\mathcal{F}}{n-1} } \bigr| \nonumber \\
= &\sum_{n=1}^{\infty} \itstep{n} \bigl| \sum_{i=1}^d \expectation{ \iterand{\vectComponent{B}{i}}{n} | \iterand{\mathcal{F}}{n-1} } ( \iterand{\vectComponent{R}{i}}{n-1} - \criticalpoint{\vectComponent{r}{i}} ) \bigr| \nonumber \\
\leq &\sum_{n=1}^{\infty} \itstep{n} \sum_{i=1}^d \bigl| \expectation{ \iterand{\vectComponent{B}{i}}{n} | \iterand{\mathcal{F}}{n-1} } \bigr| \bigl| \iterand{\vectComponent{R}{i}}{n-1} - \criticalpoint{\vectComponent{r}{i}} \bigr| \nonumber \\
\leq &\sum_{n=1}^{\infty} \itstep{n} \sum_{i=1}^d | \iterand{\vectComponent{B}{i}}{n} | \bigl( | \iterand{\vectComponent{R}{i}}{n-1} | + | \criticalpoint{\vectComponent{r}{i}} | \bigr) \nonumber \\
\leq &\sum_{n=1}^{\infty} \itstep{n} \sum_{i=1}^d | \iterand{\vectComponent{B}{i}}{n} | \bigl( | \iterand{\vectComponent{R}{i}}{0} | + | \criticalpoint{\vectComponent{r}{i}} | + \gradientConstant \sum_{m=1}^{n-1} \itstep{m} \bigr) \nonumber \\
\leq &\max_{i=1,...,d} \bigl\{ | \iterand{\vectComponent{R}{i}}{0} |
 + | \criticalpoint{\vectComponent{r}{i}} | \bigr\} \sum_{n=1}^{\infty} \itstep{n} \sum_{i=1}^d | \iterand{\vectComponent{B}{i}}{n} | \nonumber \\
&+ \gradientConstant \sum_{n=1}^{\infty} \bigl( \itstep{n} \sum_{i=1}^d | \iterand{\vectComponent{B}{i}}{n} | \sum_{m=1}^{n-1} \itstep{m} \bigr). \label{eqn:Bound_on_sum_of_Bn_times_Rn_minus_critical_point_r}
\end{align}
This expression is similar to \cite[Eq.~(36)]{SandersBorstLeeuwaarden2012}, but applies to algorithm \refEquation{eqn:Online_algorithm_for_loglikelihood_for_Rd} instead of algorithm \refEquation{eqn:Online_algorithm_for_loglikelihood_for_general_R}. Next, we turn to bounding $| \iterand{\vectComponent{B}{i}}{n} |$ from above. Expression \cite[Eq.~(37)]{SandersBorstLeeuwaarden2012} says
\begin{align}
| \iterand{\vectComponent{B}{i}}{n} |
\leq \frac{\LipschitzConstant}{2} \sum_{\svA \in \stateSpace} \expectation{ | \iterand{\vectComponent{\hat{\Pi}}{\svA}}{n} - \vectComponent{\pi}{\svA}(\iterand{\vect{R}}{n-1}) | | \iterand{\mathcal{F}}{n-1} } \label{eqn:Error_bias__Upper_bound_after_triangle_inequality}
\end{align}
and still holds in the present case, because $| \vectComponent{g}{i}(\vect{\mu}) - \vectComponent{g}{i}(\vect{\nu}) | \leq 2 \max_{\svA, i} \{ | \matrixElement{A}{\svA}{i} | \} \totalVariation{ \vect{\mu} - \vect{\nu} } = \LipschitzConstant \totalVariation{ \vect{\mu} - \vect{\nu} }$ for $i = 1, ..., d$ \cite[\S{III-B}]{SandersBorstLeeuwaarden2012}. When using algorithm \refEquation{eqn:Online_algorithm_for_loglikelihood_for_Rd}, however, we cannot proceed and apply \cite[Lemma~7]{SandersBorstLeeuwaarden2012}. Instead, we will use \refLemma{lem:Probabilistic_bound_on_measuring_the_equilibrium_distribution}, the proof of which can be found in \refAppendixSection{sec:Proof_of__Probabilistic_bound_on_measuring_the_equilibrium_distribution}. We derive \refLemma{lem:Probabilistic_bound_on_measuring_the_equilibrium_distribution} using a large deviations result in \cite{CattiauxGuillin2006}, and we note that it is related to the mixing time of the underlying stochastic process, see also \cite[Thm.~12.4]{LPW08}.

\begin{lemma}
\label{lem:Probabilistic_bound_on_measuring_the_equilibrium_distribution}
There exist $c_1$, $c_2$, $c_3$, $c_4 \in (0,\infty)$, so that for $\iterror{n} \in [0,1]$ and $\svA \in \stateSpace$,
\begin{align}
&\probability{ \bigl| \iterand{\vectComponent{\hat{\Pi}}{\svA}}{n} - \vectComponent{\pi}{\svA}(\iterand{\vect{R}}{n-1}) \bigr| \geq \iterror{n} } \label{eqn:Probability_bound_on_pi_measurement} \\
&\leq c_{1} \exp{ \bigl( c_{2} \sum_{m=1}^n \itstep{m} - c_3 \frac{ ( \iterror{n} )^2 }{ \itfreq{n} } \exp{ \bigl( - c_4 \sum_{m=1}^n \itstep{m} \bigr) } \bigr) } \nonumber.
\end{align}
\end{lemma}
\noindent Similar to \cite[Eq.~(38)]{SandersBorstLeeuwaarden2012}, we conclude that
\begin{align}
&| \iterand{\vectComponent{B}{i}}{n} |
\leq \frac{ \LipschitzConstant \cardinality{\stateSpace} }{2} \max \{ 1, c_1 \} \bigl( \iterror{n} + \label{eqn:Bound_on_abs_Bni} \\
&\exp \bigl( c_{2} \sum_{m=1}^n \itstep{m} - c_3 \frac{ ( \iterror{n} )^2 }{ \itfreq{n} } \exp{ \bigl( - c_4 \sum_{m=1}^n \itstep{m} \bigr) } \bigr) \bigr) \nonumber
\end{align}

Substituting \refEquation{eqn:Bound_on_abs_Bni} into \refEquation{eqn:Bound_on_sum_of_Bn_times_Rn_minus_critical_point_r}, we conclude that 
\begin{align}
\sum_{n=1}^{\infty} \itstep{n} \bigl| \expectation{ \transpose{ \iterand{\vect{B}}{n} } ( \iterand{\vect{R}}{n-1} - \criticalpoint{\vect{r}} ) | \iterand{\mathcal{F}}{n-1} } \bigr| < \infty
\end{align}
with probability one if $\itstep{n}$, $\iterror{n}$ and $\itfreq{n}$ are such that
\begin{align}
&\sum_{n=1}^{\infty} \itstep{n} \bigl( \sum_{m=1}^{n-1} \itstep{m} \bigr) \bigl( \iterror{n} + \\
&\exp \Bigl( c_{2} \sum_{m=1}^n \itstep{m} - c_3 \frac{ (\iterror{n})^2 }{ \itfreq{n} } \exp{ \Bigl( - c_4 \sum_{m=1}^n \itstep{m} \Bigr) } \Bigr) \bigr) < \infty, \nonumber
\end{align}
which completes the proof of \refTheorem{thm:Convergence}.
\section{Conclusion}

We have identified the achievable region for Markov processes with product-form distributions when an operator can control one or more parameters (transition rates) of the system. The shape and size of the achievable region was shown to be intimately related with the state space and the parameters that can be controlled. Loosely speaking, the larger the state space is, the more configurable parameters there are, the more performance measures there are that can be achieved.

Future work may be aimed at relating the achievable region to parameter domains. In the present work we assumed that parameters can be arbitrarily set, while in practice operators can only choose parameters from a compact set.

We also described how to use a (distributed) online algorithm from \cite{SandersBorstLeeuwaarden2012} to find parameters such that the performance measure of the system equals any target performance measure taken from the achievable region. This required broadening the scope of applicability of the online algorithm in \cite{SandersBorstLeeuwaarden2012}, because the shape of the achievable region is only known to us when transition rates are unbounded. By capitalizing on and generalizing an existing proof methodology 
\cite{SandersBorstLeeuwaarden2012}, we have provided sufficient conditions that guarantee convergence of the algorithm when an operator can arbitrarily set parameters. Because parameter values were now assumed unbounded, the conditions on the step sizes $\itstep{n}$ and observation frequencies $\itfreq{n}$ had to be more stringent in order to lessen the risk of extreme parameter growth.

Further research on the algorithm may be aimed at studying the necessity of such stringent conditions. Less stringent conditions possibly allow for increased convergence speeds. To further substantiate this, a better understanding is required of the delicate interplay between the mixing times and the convergence properties of the algorithm.

\section*{Acknowledgement}

This research was financially supported by The Netherlands Organization for Scientific Research (NWO) in the framework of the TOP-GO program and by an ERC Starting Grant.


\bibliographystyle{abbrv}
\bibliography{Bibliography}

\appendix

\subsection{Proof of \refProposition{prop:Sequences_that_meet_the_conditions}}
\label{sec:Proof_of__Sequences_that_meet_the_conditions}

First define $\iterand{L}{n} = \sum_{m=1}^n ( m \ln{(m+1)} )^{-1}$ and $\iterand{H}{n} = \sum_{m=1}^{n} m^{-1}$ .

Let $\itstep{n}$, $\iterror{n}$ and $\itfreq{n}$ be as in \textnormal{(a)}. Condition \textnormal{(i)} is satisfied, because $\sum_{n=1}^{\infty} \itstep{n} = \sum_{n=1}^{\infty} (n \ln (n+1) )^{-1} \geq \sum_{m=2}^{\infty} ( m \ln m )^{-1} = \infty$ and $\sum_{n=1}^{\infty} (\itstep{n})^2 \leq \sum_{m=1}^{\infty} m^{-2} < \infty$. Next, we examine condition \textnormal{(ii)}. Note that
\begin{align}
\sum_{n=1}^{\infty} \itstep{n} \Bigl( \sum_{m=1}^{n-1} \itstep{m} \Bigr) \iterror{n}
\leq \sum_{n=1}^{\infty} \frac{ \iterand{H}{n-1} }{ n^{1+\alpha/2} \ln{(n+1)} } \label{eqn:Bound_on_sum_an_en_sum_am_en_for_timid_step_sizes}
\end{align}
is finite, because $\iterand{H}{n-1} \leq \ln{(n-1)} + 1 \leq \ln{(n+1)} + 1$ for all $n \geq 2$ and $\alpha > 0$. We still need to check whether
\begin{align}
&\sum_{n=1}^{\infty} \itstep{n} \bigl( \sum_{m=1}^{n-1} \itstep{m} \bigr) \exp \bigl( c_{2} \sum_{m=1}^n \itstep{m}
- c_3 \frac{ (\iterror{n})^2 }{ \itfreq{n} } \nonumber \\
&\times \exp{ \bigl( - c_4 \sum_{m=1}^n \itstep{m} \bigr) } \bigr) 
= \sum_{n=1}^{\infty} \frac{ \iterand{L}{n-1} }{ n \ln{(n+1)} } \exp \bigl( c_{2} L_{n} \nonumber \\ 
&- c_3 n^{\delta - \alpha} \exp{ \bigl( - c_4 \iterand{L}{n} \bigr) } \bigr)
\label{eqn:Summation_part_with_exponential_for_timid_step_sizes}
\end{align}
is finite for all $c_2$, $c_3$, $c_4 \in \strictlyPositiveRealNumbers$. For this, we note that 
\begin{align}
&\iterand{L}{n} 
= \frac{1}{ \ln 2 } + \sum_{m=2} \frac{1}{m \ln (m+1)}
\leq \frac{1}{ \ln 2 } + \sum_{m=2} \frac{1}{m \ln m} \nonumber \\ 
&\leq \frac{3}{ 2 \ln 2 } + \int_2^n \frac{\textnormal{d}x}{\ln x} 
= \ln( \ln n ) + \frac{3}{ 2 \ln 2 } - \ln( \ln 2 ).
\end{align}
Define $c_L = 3 / ( 2 \ln 2 ) - \ln( \ln 2 )$. Consider the summand of \refEquation{eqn:Summation_part_with_exponential_for_timid_step_sizes} and upper bound it for all $n \geq 3$ by writing
\begin{align}
&\frac{ \iterand{L}{n-1} }{ n \ln{(n+1)} } \exp \bigl( c_{2} L_{n} - c_3 n^{\delta - \alpha} \exp{ \bigl( - c_4 \iterand{L}{n} \bigr) } \bigr) \nonumber \\
&\leq \frac{ \ln( \ln (n-1) ) + c_L }{ n \ln{(n+1)} } \exp \bigl( c_{2} L_{n} - c_3 n^{\delta - \alpha} \exp{ \bigl( - c_4 \iterand{L}{n} \bigr) } \bigr) \nonumber \\
&\leq \frac{1 + c_L}{ n } \exp \bigl( c_{2} L_{n} - c_3 n^{\delta - \alpha} \exp{ \bigl( - c_4 \iterand{L}{n} \bigr) } \bigr).
\end{align}
We next upper bound the exponential. Because the exponential function is monotonically increasing and $c_2$, $c_3$, $c_4 > 0$, we can upper bound again by
\begin{align}
&\leq \frac{ \e{c_2 c_L} ( 1 + c_L ) }{ n } ( \ln n )^{c_2} \exp{ \bigl( - c_3 \e{- c_4 c_L} n^{\delta - \alpha} ( \ln n )^{-c_4} \bigr) }. \nonumber
\end{align}
Now note that for any $\epsilon > 0$, $\ln n \leq n^\epsilon$ for all $n$ sufficiently large, implying that
\begin{align}
&\leq \frac{ \e{c_2 c_L} ( 1 + c_L ) }{ n } ( \ln n )^{c_2} \exp{ \bigl( - c_3 \e{c_4 c_L} n^{\delta - \alpha - c_4 \epsilon} \bigr) }
\end{align}
for all $n$ sufficiently large. Define $\varepsilon = \delta - \alpha - c_4 \epsilon$ and choose $\epsilon \in (0, ( \delta - \alpha - 1 ) / c_4 ]$, so that $\varepsilon \geq 1$. We have created a sequence that is an upper bound for \refEquation{eqn:Summation_part_with_exponential_for_timid_step_sizes} and converges absolutely, which can be seen using the ratio test,
\begin{align}
&\lim_{n \rightarrow \infty} \Bigl| \Bigl( \frac{ \ln(n+1) }{ \ln n } \Bigr)^{c_2} \frac{n}{n+1} \e{ c_3 \e{ - c_4 c_L } ( n^{\varepsilon} - (n+1)^{\varepsilon} ) } \Bigr| \nonumber \\
&=
\begin{cases}
1 & \textrm{ if } \varepsilon < 1, \\
\e{-c_3 \e{ - c_4 c_L }} & \textrm{ if } \varepsilon = 1, \\
0 & \textrm{ if } \varepsilon > 1.
\end{cases}
\end{align}

Next, let $\itstep{n}$, $\iterror{n}$ and $\itfreq{n}$ be as in \textnormal{(b)}. Verification of condition \textnormal{(i)} is immediate. To verify condition \textnormal{(ii)}, we first note that $\sum_{n=1}^{\infty} \itstep{n} \Bigl( \sum_{m=1}^{n-1} \itstep{m} \Bigr) \iterror{n} = \sum_{n=1}^{\infty} n^{-1-\alpha/2} < \infty$. What remains is to check whether
\begin{align}
&\sum_{n=1}^{\infty} \itstep{n} \bigl( \sum_{m=1}^{n-1} \itstep{m} \bigr) \exp \bigl( c_{2} \sum_{m=1}^n \itstep{m}
- c_3 \frac{ (\iterror{n})^2 }{ \itfreq{n} } \nonumber \\
&\times \exp{ \bigl( - c_4 \sum_{m=1}^n \itstep{m} \bigr) } \bigr) 
= \sum_{n=1}^{\infty} \frac{ \iterand{H}{n-1} }{ n } \exp \bigl( c_{2} \iterand{H}{n} \nonumber \\ 
&- c_3 n^{\delta-\alpha} \Bigl( \frac{ \ln n + 1 }{ \iterand{H}{n-1} } \Bigr)^2 \exp{ \bigl( - c_4 \iterand{H}{n} \bigr) } \bigr)
\end{align}
is finite. For $n \geq 2$, the summand can be upper bounded by
\begin{align}
&\frac{ \iterand{H}{n-1} }{ n } \exp \bigl( c_{2} \iterand{H}{n} - c_3 n^{\delta-\alpha} \Bigl( \frac{ \ln n + 1 }{ \iterand{H}{n-1} } \Bigr)^2 \exp{ \bigl( - c_4 \iterand{H}{n} \bigr) } \bigr) \nonumber \\ 
&\leq \e{ c_2 } n^{c_2-1} ( \ln n + 1 ) \exp{ \bigl( - c_3 \e{ - c_4 } n^{\delta - \alpha - c_4} \bigr) }.
\end{align}
After defining $\varepsilon = \delta - \alpha - c_4 \geq 1$, we conclude using the ratio test that
\begin{align}
&\lim_{n \rightarrow \infty} \Bigl| \frac{ \ln(n+1) }{ \ln n } \Bigl( \frac{n+1}{n} \Bigr)^{c_2-1} \e{ c_3 \e{- c_4} ( n^{\varepsilon} - (n+1)^{\varepsilon} ) } \Bigr| \nonumber \\
&=
\begin{cases}
1 & \textrm{ if } \varepsilon < 1, \\
\e{-c_3 \e{ - c_4 c_L }} & \textrm{ if } \varepsilon = 1, \\
0 & \textrm{ if } \varepsilon > 1.
\end{cases}
\end{align}
This proves absolute convergence, since $\varepsilon \geq 1$. \QuodEratDemonstrandum

\subsection{Proof of \refLemma{lem:Probabilistic_bound_on_measuring_the_equilibrium_distribution}}
\label{sec:Proof_of__Probabilistic_bound_on_measuring_the_equilibrium_distribution}

Consider the following setup, similar to that in \cite{DS91} for discrete-time Markov chains. Define a graph $G=(V,E)$, where $V$ denotes the vertex set in which each vertex corresponds to a state in $\stateSpace$ and $E$ denotes the set of directed edges. An edge $(\svA,\svB)$ is in $E$ if and only if $\phi(e) = \vectComponent{\pi}{\svA} \matrixElement{Q}{\svA}{\svB} = \vectComponent{\pi}{\svB} \matrixElement{Q}{\svB}{\svA} > 0$. Here, $Q$ denotes the generator matrix of $\process{X(t)}{t \geq 0}$. For every pair of distinct vertices $\svA, \svB \in \stateSpace$, choose a path $\gamma_{\svA,\svB}$ (along the edges of $G$) from $\svA$ to $\svB$. Paths may have repeated vertices but a given edge appears at most once in a given path. Let $\Gamma$ denote the collection of all paths (one for each ordered pair $\svA, \svB$). Irreducibility of $\process{X(t)}{t \geq 0}$ guarantees that such paths exists. For $\gamma_{\svA,\svB} \in \Gamma$ define the path length by $\pnorm{\gamma_{\svA,\svB}}{\phi} = \sum_{ e \in \gamma_{\svA,\svB} } \phi(e)^{-1}$. Also, let $\kappa = \max_{e} \sum_{ \{ \gamma_{\svA,\svB} \in \Gamma | e \in \gamma_{\svA,\svB} \} } \pnorm{\gamma_{\svA,\svB}}{\phi} \vectComponent{\pi}{\svA} \vectComponent{\pi}{\svB}$. 

With this set-up, the upper bound $\probability{ \bigl| \iterand{\vectComponent{\hat{\Pi}}{\svA}}{n} - \vectComponent{\pi}{\svA}(\iterand{\vect{R}}{n-1}) \bigr| \geq \iterror{n} }
\leq ( \iterand{\vectComponent{\pi}{\min}}{n} )^{-\frac{1}{2}} \exp{ ( - (\iterror{n})^2 / 4 \kappa \cardinality{ \stateSpace }^2 \itfreq{n} ) }$,
where $\iterand{\vectComponent{\pi}{\min}}{n} = \min_{\svA \in \stateSpace} \vectComponent{\pi}{\svA}(\iterand{\vect{R}}{n})$, has been derived in \cite[Appendix~D]{SandersBorstLeeuwaarden2012}. We will show that \refLemma{lem:Probabilistic_bound_on_measuring_the_equilibrium_distribution} can be derived using this upper bound, by calculating a lower bound on $\iterand{\vectComponent{\pi}{\min}}{n}$ and an upper bound on $\kappa$ when using the online algorithm in \refEquation{eqn:Online_algorithm_for_loglikelihood_for_Rd}. 

\subsubsection{Bounding $\iterand{\vectComponent{\pi}{\min}}{n}$}

By non-negativity of $\exp{(\cdot)}$,
\begin{align}
&\iterand{\vectComponent{\pi}{\min}}{n}
\geq \frac{ \min_{\svA \in \stateSpace} \bigl\{ \exp{ \vectComponent{ ( A \iterand{\vect{R}}{n} + \vect{b} ) }{\svA} } \bigr\} }{ \cardinality{\stateSpace} \max_{\svA \in \stateSpace} \bigl\{ \exp{ \vectComponent{ ( A \iterand{\vect{R}}{n} + \vect{b} ) }{\svA} } \bigr\} }.
\end{align}
Then note that
\begin{align}
&\min_{\svA \in \stateSpace} \bigl\{ \exp{ \vectComponent{ ( A \iterand{\vect{R}}{n} + \vect{b} ) }{\svA} } \bigr\} 
= \exp{ \bigl( \min_{\svA \in \stateSpace} \bigl\{ \vectComponent{ ( A \iterand{\vect{R}}{n} + \vect{b} ) }{\svA} \bigr\} \bigr) } \nonumber \\
&\geq \exp{ \Bigl( \min_{\svA \in \stateSpace} \{ \vectComponent{b}{\svA} \} + \min_{\svB \in \stateSpace} \bigl\{ \sum_{i=1}^d \matrixElement{A}{\svB}{i} \vectComponent{\iterand{R}{n}}{i} \bigr\} \Bigr) }
\geq \exp \Bigl( \min_{\svA \in \stateSpace} \{ \vectComponent{b}{\svA} \} \nonumber \\
&- \max_{\svB \in \stateSpace} \bigl\{ \sum_{i=1}^d | \matrixElement{A}{\svB}{i} | | \vectComponent{\iterand{R}{n}}{i} | \bigr\} \Bigr)
\geq \exp \Bigl( \min_{\svA \in \stateSpace} \{ \vectComponent{b}{\svA} \} \nonumber \\ 
&- \max_{i=1,...,d} \bigl\{ | \vectComponent{\iterand{R}{0}}{i} | + \gradientConstant \sum_{m=1}^{n} \itstep{m} \bigr\} \max_{\svB \in \stateSpace} \pnorm{ \matrixRow{A}{\svB} }{1} \Bigr). \label{eqn:Bound_on_min_exp_ARplusB}
\end{align}
Here, $\max_{\svB \in \stateSpace} \pnorm{ \matrixRow{A}{\svB} }{1}$ denotes the maximum absolute row sum of $A$. Similar to \refEquation{eqn:Bound_on_min_exp_ARplusB}, $\max_{\svA \in \stateSpace} \bigl\{ \exp{ \vectComponent{ ( A \iterand{\vect{R}}{n} + \vect{b} ) }{\svA} } \bigr\}$ can be upper bounded, and one finds that
\begin{align}
&\iterand{\vectComponent{\pi}{\min}}{n}
\geq \frac{1}{ \cardinality{\stateSpace} } \exp \Bigl( \min_{\svA \in \stateSpace} \{ \vectComponent{b}{\svA} \} - \max_{\svA \in \stateSpace} \{ \vectComponent{b}{\svA} \} \nonumber \\
&- 2 \max_{i=1,...,d} \bigl\{ | \vectComponent{\iterand{R}{0}}{i} | + \gradientConstant \sum_{m=1}^{n} \itstep{m} \bigr\} \max_{\svB \in \stateSpace} \pnorm{ \matrixRow{A}{\svB} }{1} \Bigr).
\end{align}


To summarize, there exist $c_1^2 = \cardinality{\stateSpace} \exp \bigl( \max_{\svA \in \stateSpace} \{ \vectComponent{b}{\svA} \} - \min_{\svA \in \stateSpace} \{ \vectComponent{b}{\svA} \} + 2 \max_{i=1,...,d} \bigl\{ | \vectComponent{\iterand{R}{0}}{i} | \bigr\} \max_{\svB \in \stateSpace} \pnorm{ \matrixRow{A}{\svB} }{1} \bigr) > 0$ and $c_2 = \gradientConstant \max_{\svB \in \stateSpace} \pnorm{ \matrixRow{A}{\svB} }{1} > 0$ so that
\begin{align}
( \iterand{\vectComponent{\pi}{\min}}{n} )^{ - \frac{1}{2} }
\leq c_1 \exp{ \bigl( c_2 \sum_{m=1}^n \itstep{m} \bigr) }.
\label{eqn:Upper_bound_on_1_over_pimin}
\end{align}

\subsubsection{Bounding $\kappa$}

From the definition of $\kappa$, it follows that
\begin{align}
&\kappa 
\leq \sum_{e \in E} \sum_{ \{ \gamma_{\svA,\svB} \in \Gamma | e \in \gamma_{\svA,\svB} \} } \pnorm{\gamma_{\svA,\svB}}{\phi}
= \sum_{\svA, \svB \in \stateSpace} \sum_{ e \in \gamma_{\svA,\svB} } \pnorm{\gamma_{\svA,\svB}}{\phi} \nonumber \\
&= \sum_{\svA, \svB \in \stateSpace} \cardinality{ \gamma_{\svA,\svB} } \pnorm{\gamma_{\svA,\svB}}{\phi}
\leq \cardinality{G}^2 \cardinality{E} \max_{\svA,\svB \in \stateSpace} \{ \pnorm{\gamma_{\svA,\svB}}{\phi} \},
\end{align}
and from the definition of the path length it follows that
\begin{align}
&\pnorm{\gamma_{\svA,\svB}}{\phi} 
= \sum_{ e = \vectInLine{ \svC, \svD } \in \gamma_{\svA,\svB} } \frac{1}{ \vectComponent{\pi}{\svC} \matrixElement{Q}{\svC}{\svD} } 
\leq \frac{|E|}{ \iterand{\vectComponent{\pi}{\min}}{n} } \max_{\svC,\svD \in \stateSpace} \Bigl\{ \frac{1}{ \matrixElement{Q}{\svC}{\svD} } \Bigr\}. \nonumber
\end{align}

Recall that $\vectComponent{\iterand{R}{n}}{i}$ corresponds to the logarithm of a rate, i.e.~$\vectComponent{\iterand{R}{n}}{i} = \ln \matrixElement{Q}{\svA}{\svB}$ for some $\svA, \svB \in \stateSpace$. It follows that for $c_{Q} = \max_{ \{ \svC,\svD \in \stateSpace | \not\exists_{i \in \{1, ..., d\}} \vectComponent{\iterand{R}{n}}{i} = \ln \matrixElement{Q}{\svC}{\svD} \} } \{ \matrixElement{Q}{\svC}{\svD}^{-1} \} < \infty$,
\begin{align}
&\max_{\svC,\svD \in \stateSpace} \Bigl\{ \frac{1}{ \matrixElement{Q}{\svC}{\svD} } \Bigr\}
\leq \max_{i = 1, ..., d} \bigl\{ \e{ - \vectComponent{\iterand{R}{n}}{i} } \bigr\} + c_{Q} \leq \max_{i = 1, ..., d} \bigl\{ \e{ | \vectComponent{\iterand{R}{n}}{i} | } \bigr\} \nonumber \\
&+ c_{Q} \leq \max_{i = 1, ..., d} \bigl\{ \e{ | \vectComponent{\iterand{R}{0}}{i} | } \bigr\} \e{ \gradientConstant \sum_{m=1}^{n} \itstep{m} } + c_{Q}.
\end{align}

Using \refEquation{eqn:Upper_bound_on_1_over_pimin}, $\kappa \leq c_1^2 \cardinality{G}^2 \cardinality{E}^2 \exp{ ( 2 c_2 \sum_{m=1}^n \itstep{m} ) } \times ( \max_{i} \{ \exp{ ( | \vectComponent{\iterand{R}{0}}{i} | ) } \} \e{ \gradientConstant \sum_{m=1}^{n} \itstep{m} } + c_{Q} ) \leq c_1^2 \cardinality{G}^2 \cardinality{E}^2 \times ( \max_{i} \{ \e{ | \vectComponent{\iterand{R}{0}}{i} | } \} + c_{Q} ) \exp{ ( ( \gradientConstant + 2 c_2 ) \sum_{m=1}^{n} \itstep{m} ) }$,
or equivalently $\kappa \leq \frac{2}{ \cardinality{\stateSpace} c_3 } \exp{ ( c_{4} \sum_{m=1}^n \itstep{m} ) }$, where 
\begin{align}
c_3
= \frac{2}{ \cardinality{\stateSpace} c_1^2 \cardinality{G}^2 \cardinality{E}^2 } \Bigl( \max_{i = 1, ..., d} \Bigl\{ \e{ | \vectComponent{\iterand{R}{0}}{i} | } \Bigr\} + c_{Q} \Bigr)^{-1}
> 0
\end{align}
and $c_{4} = \gradientConstant + 2 c_2 = \gradientConstant ( 1 + 2 \max_{\svB \in \stateSpace} \pnorm{ \matrixRow{A}{\svB} }{1} )
> 0$. Noting that $\cardinality{G} = \cardinality{\stateSpace} < \infty$ and $\cardinality{E} \leq \cardinality{\stateSpace} ( \cardinality{\stateSpace} - 1 ) / 2 < \infty$ completes the proof. \QuodEratDemonstrandum

\end{document}